\pdfoutput=1 		
\documentclass[leqno]{article}
\usepackage[normal]{phaine_style}
\usepackage[margin=1.5in]{geometry}

\tikzcdset{
  cells={font=\everymath\expandafter{\the\everymath\displaystyle}},
}

\renewcommand{\notin}{\mathop{\not\kern0.05em\smallin}}

\newcommand{\longto}{\longrightarrow}

\newcommand{\Ktheory}{\xspace{$\Kup$-theory}\xspace}

\DeclareMathOperator{\image}{im}
\DeclareMathOperator{\map}{map}

\renewcommand{\counit}{\mathrm{counit}}
\renewcommand{\unit}{\mathrm{unit}}

\newcommand{\Subfin}{\Sub_{\fin}}

\newcommand{\jlowersharpL}{\jlowersharp^{\Lup}}

\newcommand{\iplowersharp}{i_{p,\sharp}}

\newcommand{\Zren}{\Zcal\ensuremath{\textup{-ren}}}

\DeclareMathOperator{\cont}{cont}
\newcommand{\Lcont}{L^{\cont}}
\newcommand{\Kcont}{\Kup^{\cont}}

\DeclareMathOperator{\at}{at}

\newcommand{\PrLat}{\categ{Pr}^{\kern0.05em\operatorname{L},\at}}
\newcommand{\PrRat}{\categ{Pr}^{\kern0.05em\operatorname{R},\at}}
\newcommand{\PrLomega}{\categ{Pr}^{\kern0.05em\operatorname{L},\upomega}}

\newcommand{\An}{\categ{An}}
\renewcommand{\Sp}{\categ{Sp}}
\renewcommand{\Pr}{\categ{Pr}}
\DeclareMathOperator{\dual}{dual}
\DeclareMathOperator{\stable}{st}
\newcommand{\PrLst}{\PrL_{\stable}}
\newcommand{\Prdual}{\Pr^{\dual}}
\newcommand{\Catperf}{\Cat^{\perf}}
\newcommand{\CATinfty}{\categ{\textsc{Cat}}_{\infty}}

\newcommand{\Piinfty}{\Pi_{\infty}}

\DeclareMathOperator{\Cons}{Cons}

\newcommand{\ConsP}{\Cons_{P}}
\newcommand{\ConsQ}{\Cons_{Q}}

\newcommand{\THH}{\mathrm{THH}}

\newcommand{\upc}{\mathrm{c}}

\theoremstyle{definition}

\newtheorem{outline}[equation]{Outline}

\title{\Large Localizing invariants of constructible sheaves}

\author{\normalsize Qingyuan Bai \and \normalsize Peter J. Haine}

\date{\normalsize \today}

\begin{document}

\maketitle


\begin{abstract} 
	Given an open-closed decomposition of the stratifying poset, we construct a new semi-orthogonal decomposition of the \category of constructible sheaves on a stratified space admitting an exit-path \category.
	From this we obtain a direct sum decomposition of the localizing invariants of the \category of constructible sheaves. 
	Since the $ * $-pullback to the open stratum in the usual (recollement) semi-orthogonal decomposition is not strongly left adjoint, this splitting does not follow from pure sheaf theory considerations.
	Instead, the splitting crucially relies on the exodromy equivalence: it implies that on the level of constructible sheaves, the $ \ast $-pullback to a closed stratum and the $ ! $-pushforward from an open stratum admit left adjoints.
	These new functors provide an additional semi-orthogonal decomposition (with the roles of open and closed reversed) in which the relevant functors are strongly left adjoint. 
\end{abstract}

\setcounter{tocdepth}{1}

\tableofcontents


\setcounter{section}{-1}

\section{Introduction}

The goal of this paper is to compute localizing invariants of the large \category of sheaves on a sufficiently nice stratified topological space constructible with respect to the stratification.
The key tool that allows us to perform this computation is the theory of exit-path \categories \cites[\HAapp{A}]{HA}[\href{https://arxiv.org/pdf/2108.01924\#section.3}{\S3}]{MR4683160}{arXiv:2401.12825}[\href{https://arxiv.org/pdf/2308.09550\#section.3}{\S3}]{MR4749419}.
Namely, if $ (X,P) $ is a stratified space, then under mild conditions, there exists \acategory $ \Piinfty(X,P) $ called the \textit{exit-path \category} of $ (X,P) $ so that for any compactly assembled presentable \category $ \Ecal $, there is a functorial equivalence
\begin{equation}\label{eq:exodromy_equivalence}
	\ConsP(X;\Ecal) \equivalent \Fun(\Piinfty(X,P),\Ecal)
\end{equation}
between $ \Ecal $-valued constructible sheaves on $ (X,P) $ and functors $ \Piinfty(X,P) \to \Ecal $. With the help of \eqref{eq:exodromy_equivalence}, our goal is transformed into understanding localizing invariants of functor \categories, which we study in purely categorical terms.

There is an abundant supply of stratified spaces $ (X,P) $ where \eqref{eq:exodromy_equivalence} applies (hence our main results apply as well). 
Examples include Whitney stratifications of manifolds, locally finite stratifications of a real analytic manifolds by subanalytic subsets \cites[\href{https://arxiv.org/pdf/2401.12825\#equation.5.3.9}{Theorem 5.3.9}]{arXiv:2401.12825}, stratifications of the $ \RR $-points of real varieties by Zariski subsets \cites[\href{https://arxiv.org/pdf/2401.12825\#equation.5.3.13}{Theorem 5.3.13}]{arXiv:2401.12825}, among many others.
The exit-path \category often turns out to be a $ 1 $-category, and there are many cases where it can be explicitly computed.
For example, Jansen computed the exit-path \category of reductive Borel--Serre compactifications \cite[\href{https://arxiv.org/pdf/2012.10777\#theorem.4.3}{Theorem 4.3}]{MR4651892}. 

Despite the concrete goal we have in mind, we choose to work with the very general notion of \textit{exodromic stratified \topoi} (introduced in \cite{arXiv:2401.12825}) throughout the paper. 
This abstraction not only cleans up several technicalities, but also allows for applications in other contexts. 
In one direction, there are many naturally occurring stratified topological stacks for which there exists an \textit{exodromy equivalence} of the form \eqref{eq:exodromy_equivalence}.
For example, the underlying topological stack of the Deligne--Mumford--Knudsen compactification \smash{$ \overline{\Mcal}_{g,n} $} with the natural stratification by the poset of stable genus $ g $ dual graphs with $ n $ marked points \cite[\href{https://arxiv.org/pdf/2308.09551\#theorem.6.6}{Corollary 6.6} \& \href{https://arxiv.org/pdf/2308.09551\#theorem.6.7}{Theorem 6.7}]{arXiv:2308.09551}, as well as Lurie and Tanaka's moduli stack of broken lines \cite[\href{https://arxiv.org/pdf/1805.09587\#theorem.4.4.1}{Theorem 4.4.1}]{arXiv:1805.09587}.  
On the other hand, there is a parallel theory of the exodromy equivalence for étale sheaves \cites{arXiv:1807.03281}, and computations of exit-path \categories in this context \cite{arXiv:2410.06280}. 
Our formal result can be applied to compute localizing invariants of \categories of constructible sheaves arising in both of these contexts.

Now onto our results.
We begin by stating our main result for localizing invariants of functor \categories $\Fun(\Ccal,\Ecal)$ where $\Ccal$ admits a functor into a poset $P$, then relate it back to the theory of constructible sheaves. This setting is an abstraction of the fact that the exit-path \category comes equipped with natural functor to the stratifying poset $ \Piinfty(X,P) \to P $ that should be thought of as recording which stratum a point of $ X $ lies in. 

Let $ \Tcal $ be a presentable stable \category and let $ L \colon \Catperf \to \Tcal $ be a finitary localizing invariant.
We write
\begin{equation*}
	\Lcont \colon \Prdual \to \Tcal
\end{equation*}
for Efimov's extension of $ L $ to a localizing invariant defined on the \category of dualizable presentable stable \categories and strongly left adjoint functors (see \cite{arXiv:2405.12169}). 
The main combinatorial result of this paper is the following:

\begin{theorem}[{(\Cref{cor:splitting_for_functors})}]\label{intro_thm:splitting_for_functors}
	Let $ s \colon \fromto{\Ccal}{P} $ be a functor from a small \category to a poset.
	For each $ p \in P $, write $ \Ccal_p \colonequals s^{-1}(p) $.
	Then for any dualizable presentable stable \category $ \Ecal $, the functors given by left Kan extension along the inclusions $ \Ccal_p \inclusion \Ccal $ induce a natural equivalence
	\begin{equation*}
		\Lcont(\Fun(\Ccal,\Ecal)) \equivalent \Directsum_{p \in P} \Lcont(\Fun(\Ccal_p,\Ecal)) \period
	\end{equation*}
\end{theorem}

Applying \Cref{intro_thm:splitting_for_functors} when $ \Ccal $ is the exit-path \category of an exodromic stratified \topos, we deduce the following canonical direct sum decomposition for localizing invariants of the \category of constructible sheaves:

\begin{theorem}[{(\Cref{cor:splitting_for_exodromic_stratified_topoi})}]\label{thm:main_theorem_in_intro}
	Let $ (\Xcal,P) $ be an exodromic stratified \topos and $ \Ecal $ be a dualizable presentable stable \category. 
	There are natural equivalences
	\begin{equation*}
		\Lcont(\ConsP(\Xcal;\Ecal)) \equivalent \Directsum_{p \in P} \Lcont(\LC(\Xcal_p;\Ecal)) \equivalent \Directsum_{p \in P} \Lcont(\Fun(\Piinfty(\Xcal_p),\Ecal)) \period
	\end{equation*}
	Here, $ \Xcal_p $ denotes the $ p $-th stratum of $ (\Xcal,P) $, $ \LC(\Xcal_p;\Ecal) $ is the \category of locally constant $ \Ecal $-valued sheaves on $ \Xcal_p $, and $ \Piinfty(\Xcal_p) $ is the shape of the \topos $ \Xcal_p $.
\end{theorem}

\noindent In other words, localizing invariants of the \category of constructible sheaves on a stratified \topos are canonically identified with the direct sum of localizing invariants of the \categories of locally constant sheaves on each stratum.

\begin{remark}[{(why exodromy is needed)}]
	One might be tempted to think that the decomposition in \Cref{thm:main_theorem_in_intro} follows immediately from the usual recollements in sheaf theory; however, this is not the case.
	Instead, we need extra functors that are provided by the exodromy equivalence.
	
	To explain this, let us consider the simplest case, where the stratifying poset is $ \{0 < 1\} $, $ \Xcal $ is the \topos of sheaves on a real analytic manifold $ X $, and the strata $ X_0 = Z $ and $ X_1 = U $ are subanalytic subspaces. 
	Write $ i \colon Z \inclusion X $ and $ j \colon U \inclusion X $ for the inclusions.
	Then there is a recollement on the level of sheaves
	\begin{equation*}
		\begin{tikzcd}[sep=6em]
			\Sh(Z;\Ecal) \arrow[r, "\ilowerstar"'{description, xshift=-0.8em}, shift right=2ex, hooked] & \Sh(X;\Ecal) \arrow[l, "\iupperstar" {description, xshift=0.8em}] \arrow[l, "\iuppershriek", shift right=-4ex] \arrow[r, "\jupperstar"{description, xshift=0.8em}] & \Sh(U;\Ecal) \period \arrow[l, shift left=2ex, hooked', "\jlowerstar"] \arrow[l, shift right=2ex, hooked', "\jlowershriek"'{ xshift=-0.8em}]
		\end{tikzcd}
	\end{equation*}	
	One might hope to say that this recollement restricts to the subcategory of constructible sheaves and the pullback functors $ (\iupperstar,\jupperstar) \colon \Cons_{\{0<1\}}(X;\Ecal) \to \LC(Z;\Ecal) \cross \LC(U;\Ecal) $ induce an equivalence on continuous localizing invariants.
	However, there are two problems with this:
	\begin{enumerate}
		\item Without regularity assumptions on the stratification (e.g., being Whitney), the functor $ \jlowerstar $ need not preserve constructibility with respect to the given stratifications. 
		For example, if $ X = \RR $ and $ U = (0,\infty) $ the $ * $-pushforward of a nonzero locally constant sheaf on $ U $ is never constructible.

		\item Continuous localizing invariants are only functorial in \textit{strongly} left adjoint functors, i.e., left adjoints whose right adjoint preserves colimits.
		The functor $ \jupperstar $ is rarely strongly left adjoint.
	\end{enumerate}

	So we must argue differently.
	A special case of the functoriality of the exodromy equivalence implies that $ \iupperstar \colon \Cons_{\{0<1\}}(X;\Ecal) \to \LC(Z;\Ecal) $ admits an additional left adjoint \smash{$ \ilowersharp^{\upc} $}. 
	This is a new functor only defined at the level of constructible sheaves, and cannot be extracted from general sheaf theory considerations.
	Similarly, $ \jupperstar \colon \Cons_{\{0<1\}}(X;\Ecal) \to \LC(U;\Ecal) $ admits an additional right adjoint \smash{$ \jlowerstar^{\upc} $}.
	Moreover, the functors $ \iupperstar $ and $ \jupperstar $ do define the pullbacks in a recollement of $ \Cons_{\{0<1\}}(X;\Ecal) $ into $ \LC(Z;\Ecal) $ and $ \LC(U;\Ecal) $.
	However, $ \jupperstar $ is still not strongly continuous, so this is not enough to obtain the desired splitting.
	Together with this recollement, the fact that $ \iupperstar $ admits a left adjoint implies that $ \jlowershriek $ admits a left adjoint $ \jlowershriek^{\Lup} $ so that we have functors
	\begin{equation*}
		\begin{tikzcd}[sep=6em]
			\LC(Z;\Ecal) \arrow[r, "\ilowerstar"'{description, xshift=-0.8em}, shift right=2ex, hooked] \arrow[r, "\ilowersharp^{\upc}", shift left=2ex, hooked] & \Cons_{\{0<1\}}(X;\Ecal) \arrow[l, "\iupperstar" {description, xshift=0.8em}] \arrow[l, "\iuppershriek", shift right=-4ex] \arrow[r, "\jupperstar"{description, xshift=0.8em}] \arrow[r, "\jlowershriek^{\Lup}", shift right=-4ex] & \LC(U;\Ecal) \period \arrow[l, shift left=2ex, hooked', "\jlowerstar^{\upc}"] \arrow[l, shift right=2ex, hooked', "\jlowershriek"{description, xshift=-0.8em}]
		\end{tikzcd}
	\end{equation*}
	Moreover, the strongly continuous functors $ \jlowershriek^{\Lup} $ and $ \iupperstar $ define the pullbacks in a new recollement, where the roles of the open and closed are reversed.
	See \Cref{prop:flipping_recollements}.
	Using this new recollement, we are able to show that $ \ilowersharp^{\upc} $ and $ \jlowershriek $ induce the equivalence in \Cref{thm:main_theorem_in_intro}.
\end{remark}

We now explain one of the conceptual consequences of our results, namely, how to deduce the \textit{lattice conjecture} \cite[\href{https://arxiv.org/pdf/1211.7360\#prop.1.7}{Conjecture 1.7}]{MR3477639} for the \category of constructible sheaves on stratified topological spaces from known cases for \categories of local systems \cite[\href{https://arxiv.org/pdf/2102.01566\#prop.6.16}{Proposition 6.16}]{arxiv:2102.01566}.

\begin{corollary}[{(\Cref{rem:deducing_lattice_conjecture_from_strata})}]
	Let $\CC$ be the field of complex numbers and let $(X,P)$ be an exodromic stratified topological space with finite stratifying poset $P$. 
	Assume that for each $p\in P$, the \category $\LC (X_p;\Mod_\CC)^\upomega$ satisfies the lattice conjecture. 
	Then $\ConsP(X;\Mod_\CC)^\upomega$ also satisfies the lattice conjecture.
\end{corollary}

We also record a splitting result for localizing invariants of the \category $ \ConsP(\Xcal;\Ecal^\upomega) $ of constructible sheaves with compact stalks.

\begin{theorem}[{(\Cref{cor:splitting_constructible_small_category_finite_poset})}]\label{intro_thm:splitting_for_compact_stalks}
	Let $(\Xcal,P)$ be an exodromic stratified \topos where $ P $ is a finite poset, and let $\Ecal$ be a compactly generated stable \category with compact objects $\Ecal^\upomega $.
	There is a natural equivalence 
	\begin{equation*}
		L(\ConsP(\Xcal;\Ecal^\upomega)) \equivalent \Directsum_{p \in P} L(\Fun(\Piinfty(\Xcal_p),\Ecal^\upomega)) \period
	\end{equation*}
\end{theorem}

\begin{nul}
	Let us unpack what \Cref{thm:main_theorem_in_intro,intro_thm:splitting_for_compact_stalks} say in the most accessible case, where $ (X,P) $ is a real analytic manifold with a locally finite subanalytic stratification and $ \Xcal = \Sh(X) $.
	In this case, $ \Xcal_p $ is the \topos of sheaves on the $ p $-th stratum $ X_p $, and $ \Piinfty(\Xcal_p) $ is the underlying homotopy type $ \Piinfty(X_p) $ of the stratum $ X_p $.
	So the equivalences read as
	\begin{align*}
		\Lcont(\ConsP(X;\Ecal)) &\equivalent \Directsum_{p \in P} \Lcont(\Fun(\Piinfty(X_p),\Ecal)) \\ 
	\shortintertext{and}
		L(\ConsP(X;\Ecal^\upomega)) &\equivalent \Directsum_{p \in P} L(\Fun(\Piinfty(X_p),\Ecal^\upomega)) \period
	\end{align*}
\end{nul}

\begin{remark}
	When the coefficient \category $\Ecal$ is the \category of $ R $-modules for a ring $R$ and $ L $ is \Ktheory, \cref{thm:main_theorem_in_intro} is a formalization of the following useful mnemonic \cite{MO_384495}:
	\begin{quote}
		\textit{The \Ktheory of constructible sheaves gives constructible functions with values in the \Ktheory of the ring of coefficients.}
	\end{quote}
	Compare \cite[\href{https://arxiv.org/pdf/math/0610055\#page=19}{Lemma 3.3}]{MR2330165}.
\end{remark}

We further illustrate our main result with the following concrete example.

\begin{example}
	Consider the complex projective space $\PP^n$ equipped with the stratification over the poset $ [n] = \{0 < \cdots < n\}$ defined by the standard cell structure
	\begin{equation*}
		\emptyset\subset\PP^0\subset\PP^1\subset\cdots\subset\PP^n \period
	\end{equation*}
	This is a finite subanalytic stratification of a real analytic manifold.
	Since each stratum is contractible, we deduce that for any dualizable presentable stable \category $\Ecal$ we have a splitting
	\begin{equation*}
		\Lcont(\Cons_{[n]}(\PP^n;\Ecal)) 
		\equivalent \Directsum_{0\leq k \leq n} \Lcont(\Ecal)\period
	\end{equation*}
	The reader will immediately notice that there is an interesting and complicated space of exit-paths between the strata; it is a feature that, viewed from the eyes of localizing invariants, these strata seem completely detached from each other. 
\end{example}

\begin{remark}
	This paper is motivated by Beilinson's paper \cite{MR2330165}: we wanted to interpret the argument from \cite[\href{https://arxiv.org/pdf/math/0610055\#page=18}{Lemma 3.2}]{MR2330165} in the setting of large (dualizable) \categories. 
	See \cite{ZettoMSThesis} for a similar result and its application in geometric topology.
\end{remark}

\begin{remark}
	Let $ X $ be a real analytic manifold and $ \Lambda \subset \Tup^{\ast} X $ be a closed, conic, subanalytic Lagrangian. 
	Consider the \category $ \Sh_{\Lambda}(X;\Sp) $ of sheaves of spectra on $ X $ with microsupport contained in $ \Lambda $.
	If $ \Lambda $ is the union of conormals to a $ \upmu $-stratification of $ X $, then $ \Sh_{\Lambda}(X;\Sp) $ coincides with the full subcategory spanned by the constructible sheaves \cite[Proposition 8.4.1]{MR1299726}.
	So in this case, \Cref{thm:main_theorem_in_intro,intro_thm:splitting_for_compact_stalks} provide a formula for localizing invariants of $ \Sh_{\Lambda}(X;\Sp) $ and its variant with compact stalks; there is an especially nice formula for topological Hochschild homology (see \Cref{cor:K-theory_and_THH}).
	It would be very interesting to give a formula for localizing invariants of $ \Sh_{\Lambda}(X;\Sp) $ in general.
\end{remark}

\begin{outline}
	In \Cref{sec:background}, we recall some background material. 
	The aim is to fix our notations for dualizable \categories and their localizing invariants, as well as the exodromy equivalence.
	We also recall some useful properties of recollements from \cite[\HAapp{A}]{HA}.
	The familiar reader can safely skip this section. 
	Our work begins in \Cref{sec:semi-orthogonal_decompositions}, where we explain how to obtain semi-orthogonal decompositions of functor \categories. 
	Once we have such a general decomposition result, in \Cref{sec:splitting_results}, we combine these semi-orthogonal decompositions with the exodromy equivalence to prove splitting results for \categories of constructible sheaves.
\end{outline}


\begin{acknowledgments}
	We thank Yuxuan Hu and Marco Volpe for enlightening discussions around the contents of this paper.
	We thank Li He for correcting a bad typo in an earlier formulation of \Cref{prop:flipping_recollements}.

	QB was supported by the Danish National Research Foundation through the Copenhagen Centre for Geometry and Topology (DNRF151) while most of this research was conducted, and he is also grateful to Max Planck Institute for Mathematics in Bonn for its hospitality and financial support. 
	PH gratefully acknowledges support from the NSF Mathematical Sciences Postdoctoral Research Fellowship under Grant \#DMS-2102957. 
\end{acknowledgments}


\section{Background}\label{sec:background}

For the convenience of the reader, we briefly recall the basics of dualizable \categories and localizing invariants (\cref{subsec:dualizable_categories_and_localizing_invariants}) as well as exit-path \categories (\cref{subsec:exodromy}).
In \cref{subsec:recollements}, we also recall a bit about recollements and prove a few technical results that we need later on.


\subsection{Dualizable \texorpdfstring{$\infty$}{∞}-categories and localizing invariants}\label{subsec:dualizable_categories_and_localizing_invariants}

Our conventions for dualizable \categories and localizing invariants mostly follow \cite{arXiv:2405.12169}. 
We also recommend \cite{Hoyois_dualizable} for a concise presentation.

\begin{recollection}
	We write \smash{$\Catperf$} for the \category of small idempotent complete stable \categories and exact functors between them. 	
	This is a pointed \category (the category with one object and only an identity map is the intial and terminal object).
	Thus it makes sense to talk about cofiber and fiber sequences in \smash{$\Catperf$}. 
\end{recollection}

\begin{recollection}
	Let $\Tcal$ be a cocomplete stable \category. 
	A \defn{localizing invariant} (valued in $\Tcal$) is a pointed functor \smash{$ L\colon \Catperf \to \Tcal $}
	which takes cofiber-fiber sequences to fiber sequences. 	
	A localizing invariant $L$ is called \defn{finitary} if $ L $ commutes with filtered colimit in \smash{$\Catperf$} (note that the forgetful functor \smash{$\Catperf \to \Catinfty $} preserves filtered colimits).
\end{recollection}

\begin{recollection}
	We say that a functor $ \fupperstar \colon \fromto{\Ccal}{\Dcal} $ is a \defn{strongly left adjoint} if $ \fupperstar $ admits a right adjoint $ \flowerstar $ and $ \flowerstar $ is also a left adjoint.
\end{recollection}

\begin{recollection}
	We write \smash{$ \PrLst $} for the \category of presentable stable \categories and left adjoint functors.
	The \category \smash{$ \PrLst $} admits a symmetric monoidal structure given by the Lurie tensor product. 
	We use the term \defn{dualizable \category} to refer to a dualizable object of \smash{$ \PrLst $}.
	There are many equivalent characterizations of dualizable \categories; for example, a presentable stable \category $ \Ecal $ is a dualizable \category if and only if $ \Ccal $ is a retract in \smash{$ \PrLst $} of a compactly generated stable \category.

	We write \smash{$ \Prdual \subset \PrLst $} for the non-full subcategory with objects dualizable \categories and morphisms \textit{strongly} left adjoint functors.
	Let $\Tcal$ be a cocomplete stable \category. 
	Generalizing the definition for \smash{$ \Catperf $} verbatim, a functor \smash{$ L\colon \Prdual \to \Tcal $} is called a \defn{localizing invariant} if $ L $ is pointed and takes cofiber-fiber sequences to fiber sequences. 	
	Similarly, $ L $ is \defn{finitary} if $ L $ preserves filtered colimits.
\end{recollection}

\begin{recollection}
	Forming $\Ind$-completion defines a fully faithful functor 
	\begin{equation*}
		\Ind \colon \Catperf \inclusion \Prdual
	\end{equation*}
	whose image is the full subcategory spanned by the compactly generated dualizable \categories. 
	In \cite[\href{https://arxiv.org/pdf/2405.12169\#subsection.4.2}{\S4.2}]{arXiv:2405.12169}, Efimov showed a localizing invariant \smash{$ L \colon \Catperf \to \Tcal $} extends uniquely to a localizing invariant 
	\begin{equation*}
		\Lcont \colon \Prdual \to \Tcal\period
	\end{equation*} 
	We call $\Lcont$ the \defn{continuous extension} of $L$. 
	If $L$ is finitary then so is $ \Lcont $.
\end{recollection}

\begin{example}
	The functor of taking nonconnective \Ktheory $ \Kup \colon \Catperf \to \Sp $ is a finitary localizing invariant.
	Its continuous extension \smash{$\Kcont \colon \Prdual \to \Sp $} is often referred to as \textit{continuous \Ktheory}.
\end{example}


\subsection{Exodromy}\label{subsec:exodromy}

We now briefly review the theory of \textit{exodromic stratified \topoi} introduced in \cite{arXiv:2401.12825}.
The key point is that in this setting, the theory of constructible sheaves valued in a dualizable \category with $ * $-pullback functoriality reduces to the theory of copresheaves on \categories with a conservative functor to a poset with functoriality given by precomposition. 

\begin{recollection}
	Let $ P $ be a poset.
	A full subposet $ U \subset P $ is called \textit{open} if $ U $ is upwards-closed, i.e., $ p \in U $ and $ q > p $ implies that $ q \in U $.
	Dually, $ Z \subset P $ is \textit{closed} if $ Z $ is downward-closed.
\end{recollection}

\begin{recollection}[{\cite[\href{https://arxiv.org/pdf/2401.12825\#subsection.2.1}{\S2.1}]{arXiv:2401.12825}}]
	Let $ P $ be a poset.
	A \defn{$ P $-stratified \topos} is a geometric morphism $ \slowerstar \colon \Xcal\to\Fun(P,\An) $ where $\Xcal$ is an \topos.
	A \defn{stratified \topos} is an \topos stratified over some poset. 
	We often write a stratified \topos as a pair $(\Xcal,P)$, leaving the geometric morphism $ \slowerstar $ implicit.
	Given a full subposet $ S \subset P $, we write
	\begin{equation*}
		\Xcal_S \colonequals \Xcal \crosslimits_{\Fun(P,\An)} \Fun(S,\An) \comma
	\end{equation*}
	where the functor $ \Fun(S,\An) \to \Fun(P,\An) $ is given by right Kan extension along the inclusion.
	When $ S = \{p\} $ consists of a single element, we write $ \Xcal_p \colonequals \Xcal_{\{p\}} $ and refer to $ \Xcal_p $ as the \defn{$ p $-th stratum} of $ \Xcal $.

	Given a presentable \category $ \Ecal $, the \category of \defn{$ \Ecal $-valued sheaves} is the Lurie tensor product
	\begin{equation*}
		\Sh(\Xcal;\Ecal) \colonequals \Xcal \tensor \Ecal \period
	\end{equation*}
	Using the stratification $ \slowerstar \colon \Xcal \to \Fun(P,\An) $, one can define the full subcategory
	\begin{equation*}
		\ConsP(\Xcal;\Ecal) \subset \Sh(\Xcal;\Ecal)
	\end{equation*}
	of \defn{$ P $-constructible sheaves} as sheaves whose restrictions to each stratum are locally constant.
	When applied to the \topos of (hyper)sheaves on a stratified space, these definitions recover the usual ones from topology.
\end{recollection}

\begin{recollection}
	In \cite[\href{https://arxiv.org/pdf/2401.12825\#equation.2.2.10}{Definition 2.2.10}]{arXiv:2401.12825}, the authors introduce a property of a stratified \topos $ (\Xcal,P) $ called being \defn{exodromic}.
	This guarantees that there exists a small \category $ \Piinfty(\Xcal,P) $ called the exit-path \category of $ (\Xcal,P) $ together with an \defn{exodromy equivalence}
	\begin{equation*}
		\ConsP(\Xcal;\Ecal)\equivalent \Fun(\Piinfty(\Xcal,P),\Ecal)
	\end{equation*}
	for every compactly assembled presentable \category $ \Ecal $ (in particular, when $ \Ecal $ is dualizable).
	See \cite[\href{https://arxiv.org/pdf/2401.12825\#equation.2.2.10}{Definition 2.2.10} \& \href{https://arxiv.org/pdf/2401.12825\#subsection.4.1}{\S4.1}]{arXiv:2401.12825}.
	Let us enumerate the formal properties of exit-path \categories that we need in this paper:
	\begin{enumerate}
		\item The exit-path \category comes equipped with a conservative functor $ \Piinfty(\Xcal,P) \to P $.
		In particular, the fibers of this functor are anima.

		\item For each locally closed subposet $ S \subset P $, the stratified \topos $ (\Xcal_S,S) $ is exodromic and there is a natural equivalence
		\begin{equation*}
			\Piinfty(\Xcal_S,S) \equivalence \Piinfty(\Xcal,P) \cross_P S \period
		\end{equation*}
		When $ S = \{p\} $ consists of a single element, this implies that the fiber $ \Piinfty(X,P) \cross_P \{p\} $ coincides with the shape $ \Piinfty(\Xcal_p) $ of the $ p $-th statum.

		\item \textit{Functoriality:} The assignment $ (\Xcal,P) \mapsto \Piinfty(\Xcal,P) $ is functorial in all stratified morphisms between exodromic stratified \topoi and is compatible with the exodromy equivalence.
		In particular, if $ (\flowerstar,\phi) \colon (\Xcal,P) \to (\Ycal,Q) $ is a morphism of stratified \topoi and both $ (\Xcal,P) $ and $ (\Ycal,Q) $ are exodromic, then there is natural induced functor
		\begin{equation*}
			 \Piinfty(\flowerstar,\phi) \colon \Piinfty(\Xcal,P) \to \Piinfty(Y,Q) \period
		\end{equation*}
		Moreover, the exodromy equivalence fits into a commutative square
		\begin{equation*}
			\begin{tikzcd}[column sep=8em]
				\ConsQ(\Ycal;\Ecal) \arrow[r, "\fupperstar"] \arrow[d, "\wr"'{xshift=0.25ex}] & \ConsP(\Xcal;\Ecal) \arrow[d, "\wr"{xshift=-0.25ex}] \\
				\Fun(\Piinfty(\Ycal,Q),\Ecal) \arrow[r, "{-\of \Piinfty(\flowerstar,\phi)}"'] & \Fun(\Piinfty(\Xcal,P),\Ecal) \period
			\end{tikzcd}
		\end{equation*}
		As a consequence, the functor $ \fupperstar \colon \ConsQ(\Ycal;\Ecal) \to \ConsP(\Xcal;\Ecal) $ admits both a left adjoint \smash{$ \flowersharp^{\upc} $} and a right adjoint \smash{$ \flowerstar^{\upc} $}, corresponding to left and right Kan extension along $ \Piinfty(\flowerstar,\phi) $.
	\end{enumerate}
\end{recollection}


\subsection{Recollements}\label{subsec:recollements}

We now recall some background about recollements of \categories.

\begin{recollection}[\HAa{Definition}{A.8.1}]\label{rec:recollement}
	Let $ \Xcal $ be \acategory with finite limits.
	We say that functors
	\begin{equation*}
		\iupperstar \colon \fromto{\Xcal}{\Zcal} \andeq \jupperstar \colon \fromto{\Xcal}{\Ucal}
	\end{equation*}
	exhibit $ \Xcal $ as the \defn{recollement} of $ \Zcal $ and $ \Ucal $ if the following conditions hold:
	\begin{enumerate}
		\item The functors $ \iupperstar $ and $ \jupperstar $ admit fully faithful right adjoints $ \ilowerstar $ and $ \jlowerstar $, respectively.

		\item The functors $ \iupperstar \colon \fromto{\Xcal}{\Zcal} $ and $ \jupperstar \colon \fromto{\Xcal}{\Ucal} $ are left exact and jointly conservative.

		\item The functor $ \jupperstar \ilowerstar \colon \fromto{\Zcal}{\Ucal} $ is constant with value the terminal object of $ \Ucal $.
	\end{enumerate}

	We also simply say that $ (\iupperstar \colon \fromto{\Xcal}{\Zcal}, \jupperstar \colon \fromto{\Xcal}{\Ucal}) $ is a recollement to mean that $ \Xcal $ admits finite limits and $ \iupperstar $ and $ \jupperstar $ exhibit $ \Xcal $ and the recollement of $ \Zcal $ and $ \Ucal $. Be careful that this definition is not symmetric in $ \Zcal $ and $ \Ucal $.
\end{recollection}

\begin{notation}\label{ntn:extra_adjoints_in_recollements}
	Let $ (\iupperstar \colon \fromto{\Xcal}{\Zcal}, \jupperstar \colon \fromto{\Xcal}{\Ucal}) $ be a recollement.
	If $ \jupperstar $ or $ \iupperstar $ admits a left adjoint, we denote its (necessarily fully faithful) left adjoint by $ \jlowersharp \colon \incto{\Ucal}{\Xcal} $ or $ \ilowersharp \colon \incto{\Zcal}{\Xcal} $, respectively.
	If $ \ilowerstar $ admits a right adjoint, we denote its right adjoint by $ \iuppershriek \colon \fromto{\Xcal}{\Zcal} $. 
	Recall from \HAa{Remark}{A.8.5} and \HAa{Corollary}{A.8.13} that:
	\begin{enumerate}
		\item If $\Xcal$ is pointed, then $\iuppershriek$ exists.

		\item If $\Zcal$ has an initial object, then $\jlowersharp$ exists.
	\end{enumerate}
	In particular, if $ \Xcal $ is stable, both $ \Zcal $ and $ \Ucal $ are stable.
\end{notation}

\begin{nul}
	Let us also note that if $ \Xcal $ is stable, then saying that $ (\iupperstar \colon \fromto{\Xcal}{\Zcal}, \jupperstar \colon \fromto{\Xcal}{\Ucal}) $ is a recollement is equivalent to saying that $ (\image(\ilowerstar),\image(\jlowerstar)) $ form a \textit{semi-orthogonal decomposition} of $ \Xcal $.
	See \cite[\SAGsec{7.2}]{SAG}.
\end{nul}

\begin{remark}
	\Cref{ntn:extra_adjoints_in_recollements} follows the now-standard convention in six-functor formalisms of denoting the left adjoint to $ \fupperstar $ by $ \flowersharp $, if it exists.
	When a map $ f $ is étale, one usually writes $ \flowershriek $ for $ \flowersharp $; so in the setting of recollements, $ \jlowersharp $ would typically be denoted by $ \jlowershriek $.
	In this paper, we'll depart from that convention.
	The reason is that in the geometric situation we're interested in, $ \iupperstar $ and $ \jupperstar $ always admit left adjoints, and under exodromy these correspond to functors given by left Kan extension along functors at the level of exit-path \categories.
	So it is also desirable to have a uniform notation for these functors (see \Cref{ntn:left_right_Kan_extension}).
\end{remark}

\begin{recollection}\label{rec:closed_part_of_recollement_is_kernel_of_jupperstar}
	As explained in \HAa{Remark}{A.8.5}, if $ (\iupperstar \colon \fromto{\Xcal}{\Zcal}, \jupperstar \colon \fromto{\Xcal}{\Ucal}) $ is a recollement, then the functor
	\begin{equation*}
		\ilowerstar \colon \Zcal \to \ker(\jupperstar) \colonequals \setbar{X \in \Xcal}{\jupperstar(X) \equivalent \pt}
	\end{equation*}
	is an equivalence.
	Similarly, if $ \Zcal $ has an initial object $ \emptyset$, then the functor 
	\begin{equation*}
		\jlowersharp \colon \Ucal \to \ker(\iupperstar) \colonequals \setbar{X \in \Xcal}{\iupperstar(X) \equivalent \emptyset}
	\end{equation*}
	is an equivalence.
\end{recollection}

We now address the interaction between dualizability and recollements.
To do so, we need the following technical lemma.

\begin{lemma}\label{lem:i_lowerstar_preserves_weakly_contractible_colimits}
	Let $ (\iupperstar \colon \fromto{\Xcal}{\Zcal}, \jupperstar \colon \fromto{\Xcal}{\Ucal}) $ be a recollement.
	Let $ \Acal $ be a weakly contractible \category, and assume that $ \Xcal $ admits $ \Acal $-shaped colimits.
	Then $ \ilowerstar \colon \incto{\Zcal}{\Xcal} $ preserves $ \Acal $-shaped colimits.
\end{lemma}

\begin{proof}
	Let $f \colon \Acal \to \Zcal$ be a diagram in $\Zcal$ indexed by $\Acal$.
	We need to show that the natural map $\colim_{\Acal} \ilowerstar \of f \to 
	\ilowerstar (\colim_{\Acal} f)$ is an equivalence in $\Xcal$. 
	It suffices to check this map becomes an equivalence after applying $\iupperstar$ and $\jupperstar$ separately. 
	Note that $\iupperstar$ preserves colimits and $\iupperstar \ilowerstar \equivalent \id{\Zcal}$, so the first case is clear. 
	For the second case, note that $\jupperstar$ preserves colimits and $\jupperstar \ilowerstar f\equivalent \ast $ is the constant functor with value the terminal object of $\Ucal$. 
	So we need to show
	\begin{equation*}
		\textstyle \colim_{\Acal} \ast \to \ast
	\end{equation*}
	is an equivalence in $\Ucal$. 
	This follows from the assumption that $\Acal$ is weakly contractible.
\end{proof}

\begin{corollary}\label{cor:recollements_where_X_is_presentable_etc}
	Let $ (\iupperstar \colon \fromto{\Xcal}{\Zcal}, \jupperstar \colon \fromto{\Xcal}{\Ucal}) $ be a recollement.
	Then:
	\begin{enumerate}
		\item\label{cor:recollements_where_X_is_presentable_etc.1} If $ \Xcal $ is presentable, then $ \ilowerstar $ preserves weakly contractible colimits and both $ \Zcal $ and $ \Ucal $ are presentable.

		\item\label{cor:recollements_where_X_is_presentable_etc.2} If $ \Xcal $ is a dualizable \category, then $ \Zcal $ and $ \Ucal $ are both dualizable \categories.
	\end{enumerate}
\end{corollary}

\begin{proof}
	For \eqref{cor:recollements_where_X_is_presentable_etc.1}, first note that \Cref{lem:i_lowerstar_preserves_weakly_contractible_colimits} shows that $ \ilowerstar $ preserves weakly contractible colimits. 
	In particular, $ \ilowerstar $ preserves filtered colimits.
	Thus $ \Zcal $ is an $ \upomega $-accessible localization of the presentable \category $ \Xcal $, hence also presentable.
	To see that $ \Ucal $ is presentable, note that by \Cref{rec:closed_part_of_recollement_is_kernel_of_jupperstar}, there is a pullback square of \categories
	\begin{equation*}
	    \begin{tikzcd}[sep=2.25em]
	       \Ucal \arrow[dr, phantom, very near start, "\lrcorner", xshift=-0.25em, yshift=0.12em] \arrow[d] \arrow[r, "\jlowersharp", hooked] & \Xcal  \arrow[d, "\iupperstar"] \\ 
	       \pt \arrow[r, "\emptyset"', hooked] & \Zcal \comma
	    \end{tikzcd}
	\end{equation*}
	where the bottom horizontal functor picks out the initial object of $ \Zcal $.
	Note that $ \Xcal $ and $ \Zcal $ are presentable and $ \emptyset \colon \fromto{\pt}{\Zcal} $ and $ \iupperstar $ are left adjoints.
	The presentability of $ \Ucal $ thus follows from the fact that the forgetful functor \smash{$ \fromto{\PrL}{\CATinfty} $} preserves limits \HTT{Proposition}{5.5.3.13}.

	For \eqref{cor:recollements_where_X_is_presentable_etc.2}, first note that since $ \Xcal $ is stable, $ \ilowerstar $ admits a right adjoint $ \iuppershriek $ and both $ \Zcal $ and $ \Ucal $ are stable (see \Cref{ntn:extra_adjoints_in_recollements}).
	By \eqref{cor:recollements_where_X_is_presentable_etc.1}, both $ \Zcal $ and $ \Ucal $ are also presentable.
	Note that since $ \ilowerstar $ and $ \jlowersharp $ are fully faithful, we have $ \iupperstar \ilowerstar \equivalent \id{\Zcal} $ and $ \jupperstar \jlowersharp \equivalent \id{\Ucal} $.
	Since $ \iupperstar $, $ \ilowerstar $, $ \jupperstar $, and $ \jlowersharp $ are all left adjoints, $ \Zcal $ and $ \Ucal $ are retracts of $ \Xcal $ in \smash{$ \PrLst $}.
	Since $ \Xcal $ is dualizable, we deduce that $ \Zcal $ and $ \Ucal $ are also dualizable. 
\end{proof}

We now recall the important fact that continuous localizing invariants split recollements with the property that $ \jupperstar $ is strongly left adjoint.

\begin{lemma}[{\cite[\href{https://arxiv.org/pdf/2405.12169\#theo.4.6}{Proposition 4.6} \& \href{https://arxiv.org/pdf/2405.12169\#theo.1.76}{Remark 1.76}]{arXiv:2405.12169}}]\label{lem:continuous_localizing_invariants_split_recollements_where_jupperstar_is_strongly_left_adjoint}
	Let $ \Xcal $ be a dualizable \category and let $ (\iupperstar \colon \fromto{\Xcal}{\Zcal}, \jupperstar \colon \fromto{\Xcal}{\Ucal}) $ be a recollement.
	If $ \jupperstar $ is strongly left adjoint, then the induced map
	\begin{equation*}
		(\Lcont(\iupperstar),\Lcont(\jupperstar)) \colon \Lcont(\Xcal) \longto \Lcont(\Zcal) \directsum \Lcont(\Ucal)
	\end{equation*}
	is an equivalence. 
\end{lemma}

\begin{proof}
	Set $ (i_1,i_2) = (i_*,\jlowersharp) $ in {\cite[\href{https://arxiv.org/pdf/2405.12169\#theo.1.76}{Remark 1.76}]{arXiv:2405.12169}}.
	The conditions there are easily verified; we omit the details.
\end{proof}


\section{Semi-orthogonal decompositions}\label{sec:semi-orthogonal_decompositions}

Let $ s \colon \Ccal \to P $ be a functor from \acategory to a poset, and let $ Z \subset P $ be a closed subposet with open complement $ U = P \sminus Z $.
In this section we show that for any presentable stable \category $ \Ecal $, the functor \category $ \Fun(\Ccal,\Ecal) $ decomposes as a recollement of $ \Fun(\Ccal \cross_P Z,\Ecal) $ and $ \Fun(\Ccal \cross_P U,\Ecal) $.
See \Cref{lem:existence_of_recollements_for_functors_out_of_a_layered_category}.
We also explain why this implies that the $ \sharp $-pushforward from the open piece admits an additional left adjoint, and there is another recollement decomposing $ \Fun(\Ccal,\Ecal) $ into $ \Fun(\Ccal \cross_P U,\Ecal) $ and $ \Fun(\Ccal \cross_P Z,\Ecal) $, with the roles of the open and closed swapped.
See \Cref{prop:flipping_recollements,ex:recollement_for_functors_out_of_a_layered_category}.

We begin by fixing some general notation.

\begin{notation}\label{ntn:left_right_Kan_extension}
	Let $ f \colon \Ccal \to \Dcal $ be a functor between \categories, and let $ \Ecal $ be \acategory.
	We write 
	\begin{equation*}
		\fupperstar \colon \Fun(\Dcal,\Ecal) \to \Fun(\Ccal,\Ecal)
	\end{equation*}
	for the functor given by precomposition with $ f $.
	If $ \fupperstar $ admits a left adjoint, we denote it by $ \flowersharp $, and if $ \fupperstar $ admits a right adjoint, we denote it by $ \flowerstar $.
\end{notation}

\begin{notation}
	Let $ s \colon \fromto{\Ccal}{P} $ be a functor from a small \category to a poset.
	Given a subposet $ S \subset P $, we write $ \Ccal_S \colonequals \Ccal \cross_P S $.
	For $ p \in P $, we simply write $ \Ccal_p \colonequals \Ccal_{\{p\}} $.
\end{notation}

\begin{convention}
	Let $ s \colon \fromto{\Ccal}{P} $ be a functor from a small \category to a poset, and let $ Z \subset P $ be a closed subposet with open complement $ U \colonequals P \sminus Z $.
	We write $ i \colon \incto{\Ccal_Z}{\Ccal} $ and $ j \colon \incto{\Ccal_U}{\Ccal} $  for the inclusions.
\end{convention}

Now for our recollement description of $ \Fun(\Ccal,\Ecal) $.
Since we want this result to be as widely applicable as possible (e.g., when $ \Ecal $ is small or not stable), we've stated this result with minimal assumptions on $ \Ecal $. 

\begin{lemma}\label{lem:existence_of_recollements_for_functors_out_of_a_layered_category}
	Let $ s \colon \fromto{\Ccal}{P} $ be a functor from a small \category to a poset, and let $ Z \subset P $ be a closed subposet with open complement $ U \colonequals P \sminus Z $.
	Let $ \Ecal $ be \acategory.
	Then:
	\begin{enumerate}
		\item\label{lem:existence_of_recollements_for_functors_out_of_a_layered_category.1} The pullback functors $ \iupperstar \colon \fromto{\Fun(\Ccal,\Ecal)}{\Fun(\Ccal_Z,\Ecal)} $ and $ \jupperstar \colon \fromto{\Fun(\Ccal,\Ecal)}{\Fun(\Ccal_U,\Ecal)}  $ are jointly conservative.

		\item\label{lem:existence_of_recollements_for_functors_out_of_a_layered_category.2} If $ \Ecal $ admits a terminal object $ \ast $, then $ \iupperstar $ admits a fully faithful right adjoint $ \ilowerstar $ given by
		\begin{equation*}
			\ilowerstar(F)(c) = \begin{cases}
				F(c)\comma & c \in \Ccal_Z \\ 
				\pt\comma & c \notin \Ccal_Z \period
			\end{cases}
		\end{equation*}
		In particular, $ \jupperstar \ilowerstar \colon \fromto{\Fun(\Ccal_Z,\Ecal)}{\Fun(\Ccal_U,\Ecal)} $ is constant with value the terminal object.

		\item\label{lem:existence_of_recollements_for_functors_out_of_a_layered_category.3} If $ \Ecal $ admits an initial object $ \emptyset $, then $ \jupperstar $ admits a fully faithful left adjoint $ \jlowersharp $ given by
		\begin{equation*}
			\jlowersharp(F)(c) = \begin{cases}
				F(c) \comma & c \in \Ccal_U \\ 
				\emptyset \comma & c \notin \Ccal_U \period
			\end{cases}
		\end{equation*}
		In particular, $ \iupperstar \jlowersharp \colon \fromto{\Fun(\Ccal_U,\Ecal)}{\Fun(\Ccal_Z,\Ecal)} $ is constant with value the initial object.

		\item\label{lem:existence_of_recollements_for_functors_out_of_a_layered_category.4} Assume that $ \Ecal $ admits finite limits and $ \jupperstar \colon \fromto{\Fun(\Ccal,\Ecal)}{\Fun(\Ccal_U,\Ecal)} $ admits a fully faithful right adjoint $ \jlowerstar $.
		Then the functors 
		\begin{equation*}
			\iupperstar \colon \fromto{\Fun(\Ccal,\Ecal)}{\Fun(\Ccal_Z,\Ecal)} \andeq \jupperstar \colon \fromto{\Fun(\Ccal,\Ecal)}{\Fun(\Ccal_U,\Ecal)}
		\end{equation*}
		exhibit $ \Fun(\Ccal,\Ecal) $ as the recollement of $ \Fun(\Ccal_Z,\Ecal) $ and $ \Fun(\Ccal_U,\Ecal) $.
	\end{enumerate}
\end{lemma}

\begin{proof}
	The first item follows from the fact that equivalences between functors can be checked pointwise. 
	The next two items follow from the formulas for pointwise Kan extensions. 
	Item \eqref{lem:existence_of_recollements_for_functors_out_of_a_layered_category.4} is immediate from the previous items.
\end{proof}

We now explain why given a recollement of stable \categories, the functor $ \iupperstar $ admits a left adjoint if and only if $ \jlowersharp $ also admits a left adjoint.
Moreover, these extra adjoints give rise to a new recollement with the roles of the open and closed pieces swapped.
We start with a convenient lemma.

\begin{lemma}\label{lem:cofibers_of_adjoints}
	Let $ \Ccal $ and $ \Dcal $ be stable \categories and let $ R,R' \colon \Ccal \to \Dcal $ be right adjoint functors with left adjoints $ L $ and $ L' $, respectively.
	Then for any natural transformation $ \alpha \colon R \to R' $ with corresponding natural transformation $ \alphabar \colon L' \to L $:
	\begin{enumerate}
		\item\label{lem:cofibers_of_adjoints.1} The functor $ \cofib(\alphabar \colon L' \to L) $ is left adjoint to $ \fib(\alpha \colon R \to R') $.

		\item\label{lem:cofibers_of_adjoints.2} The functor $ \fib(\alphabar \colon L' \to L) $ is left adjoint to $ \cofib(\alpha \colon R \to R') $.
	\end{enumerate}
\end{lemma}

\begin{proof}
	To prove \eqref{lem:cofibers_of_adjoints.1}, let $ c \in \Ccal $ and $ d \in \Dcal $.
	We compute
	\begin{align*}
		\Map_{\Ccal}(\cofib(\alphabar \colon L'(d) \to L(d)),c) &\equivalent \fib \big\lparen\!\!
		\begin{tikzcd}[sep=3em, ampersand replacement = \&]
			\Map_{\Ccal}(L(d),c) \arrow[r, "-\of\alphabar"] \& \Map_{\Ccal}(L'(d),c))
		\end{tikzcd}
		\!\!\big\rparen \\ 
		&\equivalent \fib \big\lparen\!\!
		\begin{tikzcd}[sep=3em, ampersand replacement = \&]
			\Map_{\Dcal}(d,R(c)) \arrow[r, "\alpha\of-"] \& \Map_{\Dcal}(d,R'(c)))
		\end{tikzcd}
		\!\!\big\rparen \\ 
		&\equivalent \Map_{\Dcal}(d, \fib(\alpha \colon R(c) \to R'(c))) \comma
	\end{align*}
	as desired.

	To prove \eqref{lem:cofibers_of_adjoints.2}, first note that since $ \Dcal $ is stable, $ \Fun(\Ccal,\Dcal) $ is stable.
	Also note that for an adjunction $ G \leftadjoint F $ between stable \categories, we have $ G[-1] \leftadjoint F[1] $.
	Moreover for any map $ f $ in a stable \category we have natural equivalences
	\begin{equation*}
		\fib(f) \equivalent \cofib(f)[-1] \andeq \cofib(f) \equivalent \fib(f)[1] \period
	\end{equation*}
	Hence \eqref{lem:cofibers_of_adjoints.2} follows from \eqref{lem:cofibers_of_adjoints.1} applied to the adjunctions $ L[-1] \leftadjoint R[1] $ and $ L'[-1] \leftadjoint R'[1] $.
\end{proof}

\begin{proposition}\label{prop:flipping_recollements}
	Let $ \Xcal $ be a stable \category and let $ (\iupperstar \colon \fromto{\Xcal}{\Zcal}, \jupperstar \colon \fromto{\Xcal}{\Ucal}) $ be a recollement.
	Then:
	\begin{enumerate}
		\item\label{prop:flipping_recollements.2} If $ \iupperstar $ admits a left adjoint $ \ilowersharp \colon \incto{\Zcal}{\Xcal} $, then $ \jlowersharp $ admits a left adjoint $ \jlowersharpL \colon \fromto{\Xcal}{\Ucal} $ defined by the formula
		\begin{equation*}
			\jlowersharpL \colonequals \cofib \big\lparen\!
			\begin{tikzcd}[sep=4em]
				\jupperstar \ilowersharp \iupperstar \arrow[r, "\jupperstar \counit"] & \jupperstar
			\end{tikzcd}
			\!\big\rparen \period
		\end{equation*}

		\item\label{prop:flipping_recollements.3} If $ \jlowersharp $ admits a left adjoint $ \jlowersharpL \colon \fromto{\Xcal}{\Ucal} $, then $ \iupperstar $ admits a left adjoint $ \ilowersharp \colon \fromto{\Zcal}{\Xcal} $ defined by the formula 
		\begin{equation*}
			\ilowersharp \colonequals
			\fib \big\lparen\!
			\begin{tikzcd}[sep=4em]
				\ilowerstar \arrow[r, "\unit\,\ilowerstar"] &  \jlowersharp \jlowersharpL \ilowerstar 
			\end{tikzcd}
			\!\big\rparen \period
		\end{equation*}

		\item\label{prop:flipping_recollements.1} The functor $ \iupperstar $ admits a left adjoint if and only if $ \jlowersharp $ admits a left adjoint.

		\item\label{prop:flipping_recollements.4} If $ \iupperstar $ and $ \jlowersharp $ admit left adjoints, then the functors $ \jlowersharpL \colon \fromto{\Xcal}{\Ucal} $ and $ \iupperstar \colon \fromto{\Xcal}{\Zcal} $ exhibit $ \Xcal $ as the recollement of $ \Zcal $ and $ \Ucal $.
	\end{enumerate}
\end{proposition}

\begin{proof}
	For \eqref{prop:flipping_recollements.2}, we apply \Cref{lem:cofibers_of_adjoints} to the standard equivalence
	\begin{equation*}
		\jlowersharp \equivalent 
		\fib \big\lparen\!
			\begin{tikzcd}[sep=4em]
				\jlowerstar \arrow[r, "\unit\jlowerstar"] &   \ilowerstar \iupperstar \jlowerstar 
			\end{tikzcd}
		\!\big\rparen
	\end{equation*}
	coming from the recollement \cite[\href{https://arxiv.org/pdf/1909.03920\#nul.1.17}{1.17}]{arXiv:1909.03920}. 
	Similarly, for \eqref{prop:flipping_recollements.3}, note that the recollement provides a cofiber sequence
	\begin{equation*}
		\begin{tikzcd}[sep=4em]
			\jlowersharp \jupperstar \arrow[r, "\counit"] & \id{\Xcal} \arrow[r, "\unit"] &  \ilowerstar \iupperstar
		\end{tikzcd}
	\end{equation*}
	\cite[\href{https://arxiv.org/pdf/1909.03920\#nul.1.17}{1.17}]{arXiv:1909.03920}.
	Applying $ \iuppershriek $ to this cofiber sequence and using that $ \ilowerstar $ is fully faithful, we deduce that 
	\begin{equation}\label{eq:cofiber_formula_for_iupperstar}
		\cofib \big\lparen\!
		\begin{tikzcd}[sep=4em]
			\iuppershriek \jlowersharp \jupperstar \arrow[r, "\iuppershriek \counit"] & \iuppershriek 
		\end{tikzcd}
		\!\big\rparen
		\equivalent \iuppershriek \ilowerstar \iupperstar \equivalent \iupperstar \period
	\end{equation}
	So the claim follows by applying \Cref{lem:cofibers_of_adjoints} to the equivalence \eqref{eq:cofiber_formula_for_iupperstar}.

	Item \eqref{prop:flipping_recollements.1} is then immediate from items \eqref{prop:flipping_recollements.2} and \eqref{prop:flipping_recollements.3}. 
	To prove \eqref{prop:flipping_recollements.4}, we verify the conditions in \cref{rec:recollement}:
	\begin{enumerate}
		\item[\textbullet] The functor \smash{$ \jlowersharpL $} has a fully faithful right adjoint $ \jlowersharp $. 
		The functor $ \iupperstar $ has a fully faithful right adjoint $ \ilowerstar $.

		\item[\textbullet] Since $ \Xcal $ is stable and both \smash{$ \jlowersharpL $} and $\iupperstar$ are left adjoints, they are exact. 
		Now we show that they are jointly conservative. 
		Given an object $ x\in\Xcal $ such that both $ \iupperstar(x) =0 $ and \smash{$ \jlowersharpL(x) = 0 $} we will show that $x=0$. 
		Since $ \iupperstar(x) = 0$, we know that $ x\in \image(\jlowersharp)$. 
		So we may write $ x=\jlowersharp(u) $ for some $ u\in\Ucal $. 
		It follows that
		\begin{equation*}
			0 \equivalent \jlowersharpL(x) \equivalent \jlowersharpL \jlowersharp(u) \equivalent u \period
		\end{equation*}
		So $x=0$, as desired.

		\item[\textbullet] The functor $ \iupperstar\jlowersharp $ is left adjoint to $ \jupperstar\ilowerstar=0 $, so itself has to be the zero functor.\qedhere
	\end{enumerate} 
\end{proof}

Here's our main example of when \Cref{prop:flipping_recollements} applies:
	
\begin{example}\label{ex:recollement_for_functors_out_of_a_layered_category}
	Let $ s \colon \fromto{\Ccal}{P} $ be a functor from a small \category to a poset, and let $ Z \subset P $ be a closed subposet with open complement $ U \colonequals P \sminus Z $.
	Let $ \Ecal $ be a stable presentable \category.
	Combining \Cref{lem:existence_of_recollements_for_functors_out_of_a_layered_category,prop:flipping_recollements}, we deduce that there are functors
	\begin{equation*}
		\begin{tikzcd}[sep=6em]
			\Fun(\Ccal_Z,\Ecal) \arrow[r, "\ilowerstar"'{description, xshift=-0.8em}, shift right=2ex, hooked] \arrow[r, "\ilowersharp", shift left=2ex, hooked] & \Fun(\Ccal,\Ecal) \arrow[l, "\iupperstar" {description, xshift=0.8em}] \arrow[l, "\iuppershriek", shift right=-4ex] \arrow[r, "\jupperstar"{description, xshift=0.8em}] \arrow[r, "\jlowersharpL", shift right=-4ex] & \Fun(\Ccal_U,\Ecal) \period \arrow[l, shift left=2ex, hooked', "\jlowerstar"] \arrow[l, shift right=2ex, hooked', "\jlowersharp"{description, xshift=-0.8em}]
		\end{tikzcd}
	\end{equation*}
	Here, all functors lie above their right adjoints.
	Moreover $ \iupperstar $ and $ \jupperstar $ exhibit $ \Fun(\Ccal,\Ecal) $ as the recollement of $ \Fun(\Ccal_Z,\Ecal) $ and $ \Fun(\Ccal_U,\Ecal) $ and \smash{$ \jlowersharpL $} and $ \iupperstar $ exhibit $ \Fun(\Ccal,\Ecal) $ as the recollement of $ \Fun(\Ccal_U,\Ecal) $ and $ \Fun(\Ccal_Z,\Ecal) $.
\end{example}


\section{Splitting results}\label{sec:splitting_results}

We now combine our previous results on semi-orthogonal decompositions to deduce splitting results for localizing invariants of \categories of constructible sheaves.
In \cref{subsec:large_categories}, we treat large \categories of constructible sheaves.
In \cref{subsec:small_categories}, we treat small \categories of constructible sheaves whose stalks are also compact.

\begin{notation}
	Throughout this section, we fix a localizing invariant $ L $ for $\Catperf$.
	We write $ \Lcont $ for the continuous extension of $ L $.
\end{notation}


\subsection{Large \texorpdfstring{$\infty$}{∞}-categories}\label{subsec:large_categories}

We start by using \Cref{lem:continuous_localizing_invariants_split_recollements_where_jupperstar_is_strongly_left_adjoint,prop:flipping_recollements} to prove the following general splitting result.

\begin{proposition}\label{prop:spliting_recollements_with_extra_adjoints}
	Let $ \Xcal $ be a dualizable \category and let $ (\iupperstar \colon \fromto{\Xcal}{\Zcal}, \jupperstar \colon \fromto{\Xcal}{\Ucal}) $ be a recollement.
	If $ \jlowersharp $ admits a left adjoint $ \jlowersharpL $, then the maps
	\begin{align*}
		(\Lcont(\jlowersharpL),\Lcont(\iupperstar)) &\colon \Lcont(\Xcal) \longto \Lcont(\Ucal) \directsum \Lcont(\Zcal) \\ 
	\shortintertext{and}
		(\Lcont(\jlowersharp),\Lcont(\ilowersharp)) &\colon \Lcont(\Ucal) \directsum \Lcont(\Zcal) \longto \Lcont(\Xcal)
	\end{align*}
	are inverse equivalences.
\end{proposition}

\begin{proof}
	By \Cref{prop:flipping_recollements}, the pair \smash{$ (\jlowersharpL \colon \fromto{\Xcal}{\Ucal}, \iupperstar \colon \fromto{\Xcal}{\Zcal}) $} forms a recollement.
	Moreover, $ \iupperstar $ is strongly left adjoint.
	The fact that the top map is an equivalence thus follows from \Cref{lem:continuous_localizing_invariants_split_recollements_where_jupperstar_is_strongly_left_adjoint}.

	For the bottom map, note that since $ \jlowersharp $ and $ \ilowersharp $ are fully faithful and adjoint to \smash{$ \jlowersharpL $} and $ \iupperstar $, respectively, there are equivalences
	\begin{equation*}
		\jlowersharpL \jlowersharp \equivalent \id{\Ucal} \andeq \iupperstar \ilowersharp \equivalent \id{\Zcal} \period
	\end{equation*}
	Hence 
	\begin{equation*}
		(\Lcont(\jlowersharpL),\Lcont(\iupperstar)) \of (\Lcont(\jlowersharp),\Lcont(\ilowersharp)) \equivalent \id{\Lcont(\Ucal) \directsum \Lcont(\Zcal)} \period
	\end{equation*}
	Since \smash{$ (\Lcont(\jlowersharpL),\Lcont(\iupperstar)) $} is an equivalence, we deduce that \smash{$ (\Lcont(\jlowersharp),\Lcont(\ilowersharp)) $} is also an equivalence and these maps are inverses.
\end{proof}

In the setting of an exodromic stratified \topos with finite stratifying poset, \Cref{prop:spliting_recollements_with_extra_adjoints} implies splitting results for the \category of constructible objects.
In order to deal with infinite stratifying posets, we need a few continuity results. 
The following results can also be phrased in terms of \textit{$ P $-indexed semi-orthogonal decompositions} as in {\cite[\href{https://arxiv.org/pdf/2405.12169\#theo.1.80}{Definition 1.80} \& \href{https://arxiv.org/pdf/2405.12169\#theo.4.14}{Proposition 4.14}]{arXiv:2405.12169}}.
We include this material because it is straightforward from what we have done so far.

\begin{lemma}\label{lem:presheaves_with_values_in_a_dualizable_category_preserves_colimits}
	Let $ \Ecal $ be a presentable \category.
	Then:
	\begin{enumerate}
		\item\label{lem:presheaves_with_values_in_a_dualizable_category_preserves_colimits.1} The functor $ \Fun(-,\Ecal) \colon \fromto{\Catinfty}{\PrL} $
		with functoriality given by left Kan extension preserves colimits.

		\item\label{lem:presheaves_with_values_in_a_dualizable_category_preserves_colimits.2} If $ \Ecal $ is stable, the functor $ \Fun(-,\Ecal) \colon \fromto{\Catinfty}{\PrLst} $
		with functoriality given by left Kan extension preserves colimits.

		\item\label{lem:presheaves_with_values_in_a_dualizable_category_preserves_colimits.3} If $ \Ecal $ is dualizable, the functor $ \Fun(-,\Ecal) \colon \fromto{\Catinfty}{\Prdual} $
		with functoriality given by left Kan extension preserves colimits. 
	\end{enumerate}
\end{lemma}

\begin{proof}
	For \eqref{lem:presheaves_with_values_in_a_dualizable_category_preserves_colimits.1}, note that the functor
	\begin{equation*}
		\Fun(-,\Ecal) \colon \fromto{\Catinfty^{\op}}{\CATinfty}
	\end{equation*}
	with pullback functoriality preserves limits.
	Since the forgetful functor $ \fromto{\PrR}{\CATinfty} $ preserves limits \HTT{Theorem}{5.5.3.18},
	we see that the functor
	\begin{equation*}
		\Fun(-,\Ecal) \colon \fromto{\Catinfty^{\op}}{\PrR}
	\end{equation*}
	with pullback functoriality also preserves limits.
	Passing to left adjoints we deduce the claim.

	Item \eqref{lem:presheaves_with_values_in_a_dualizable_category_preserves_colimits.2} follows from \eqref{lem:presheaves_with_values_in_a_dualizable_category_preserves_colimits.1} and the fact that the forgetful functor $ \fromto{\PrLst}{\PrL} $ preserves colimits.
	Item \eqref{lem:presheaves_with_values_in_a_dualizable_category_preserves_colimits.3} follows from \eqref{lem:presheaves_with_values_in_a_dualizable_category_preserves_colimits.2} and the fact that the forgetful functor $ \fromto{\Prdual}{\PrLst} $ preserves colimits \cite[\href{https://arxiv.org/pdf/2405.12169\#theo.1.65}{Proposition 1.65}]{arXiv:2405.12169}.
\end{proof}

\begin{notation}
	Let $ P $ be a poset.
	Let $ \Subfin(P) $ denote the poset of \textit{finite} subposets of $ P $ ordered by inclusion.
\end{notation}

\begin{corollary}\label{cor:colimit_over_finite_subposets}
	Let $ s \colon \fromto{\Ccal}{P} $ be a functor from a small \category to a poset and let $ \Ecal $ be a dualizable \category. 
	\begin{enumerate}[label=\stlabel{cor:colimit_over_finite_subposets}, ref=\arabic*]
		\item\label{cor:colimit_over_finite_subposets.1} The natural functor $ \displaystyle \fromto{\colim_{Q \in \Subfin(P)} \Ccal_Q}{\Ccal} $ is an equivalence of \categories.

		\item\label{cor:colimit_over_finite_subposets.2} The natural functor
		\begin{equation*}
			\fromto{\colim_{Q \in \Subfin(P)} \Fun(\Ccal_Q,\Ecal)}{\Fun(\Ccal,\Ecal)}
		\end{equation*}
		is an equivalence.
		Here, the functoriality is given by left Kan extension and the colimit is computed in $ \Prdual $.
	\end{enumerate}
\end{corollary}

\begin{proof}
	For \eqref{cor:colimit_over_finite_subposets.1}, first observe that the natural functor
	\begin{equation*}
		\fromto{\colim_{Q \in \Subfin(P)} Q}{P}
	\end{equation*}
	is an equivalence, where the colimit is computed in $ \Catinfty $. 
	Hence the fact that that filtered colimits are left exact in $ \Catinfty $ implies the claim.
	Item \eqref{cor:colimit_over_finite_subposets.2} follows from \eqref{cor:colimit_over_finite_subposets.1} and \Cref{lem:presheaves_with_values_in_a_dualizable_category_preserves_colimits}.
\end{proof}

Now we can deduce our splitting results.

\begin{notation}
	For the following, given a functor $ \Ccal \to P $ from \acategory to a poset and a subposet $ Q \subset P $, write $ i_Q \colon \Ccal_Q \inclusion \Ccal $ for the inclusion
\end{notation}

\begin{corollary}\label{cor:splitting_for_functors}
	Let $ \Ecal $ be a dualizable \category. 
	Let $ s \colon \fromto{\Ccal}{P} $ be a functor from a small \category to a poset.
	If $ P $ is finite or $ L $ is finitary, then the natural map
	\begin{equation*}
		(\Lcont(\iplowersharp))_{p \in P} \colon \Directsum_{p \in P} \Lcont(\Fun(\Ccal_p,\Ecal)) \longto \Lcont(\Fun(\Ccal,\Ecal)) 
	\end{equation*}
	is an equivalence.
\end{corollary}

\begin{proof}
	First we treat the case where $ P $ is finite.
	We prove the claim by induction on the Krull dimension of $ P $.
	If $ \dim(P) = 0 $, then $ P $ is just a finite set; the claim then follows from the fact that localizing invariants preserve direct sums.
	For the induction step, let $ Z \subset P $ denote the subset of minimal elements, and let $ U \colonequals P \sminus Z $.
	Then $ Z $ is closed and $ \dim(Z) = 0 $ and $ U $ is open and $ \dim(U) = \dim(P) - 1 $.
	Hence it suffices to show that the natural map
	\begin{equation*}
		(\Lcont(i_{Z,\sharp}),\Lcont(i_{U,\sharp})) \colon \Lcont(\Fun(\Ccal_Z,\Ecal)) \directsum \Lcont(\Fun(\Ccal_U,\Ecal)) \longto \Lcont(\Fun(\Ccal,\Ecal))
	\end{equation*}
	is an equivalence.
	By \Cref{ex:recollement_for_functors_out_of_a_layered_category}, this is a special case of \Cref{prop:spliting_recollements_with_extra_adjoints} applied to $ \Xcal = \Fun(\Ccal,\Ecal) $, $ \Zcal = \Fun(\Ccal_Z,\Ecal) $, and $ \Ucal = \Fun(\Ccal_U,\Ecal) $.

	Now we treat the case where $ L $ is finitary.
	Using the finitaryness of $ L $, the case of a finite stratifying poset, and \Cref{cor:colimit_over_finite_subposets} we compute
	\begin{align*}
		\Lcont(\Fun(\Ccal,\Ecal)) &\equivalent \Lcont\paren{\colim_{Q \in \Subfin(P)} \Fun(\Ccal_Q,\Ecal)} \\ 
		&\equivalent \colim_{Q \in \Subfin(P)} \Lcont(\Fun(\Ccal_Q,\Ecal)) \\ 
		&\equivalent \colim_{Q \in \Subfin(P)} \Directsum_{q \in Q} \Lcont(\Fun(\Ccal_q,\Ecal)) \\
		&\equivalent \Directsum_{p \in P} \Lcont(\Fun(\Ccal_p,\Ecal)) \period
	\end{align*}
	Finally, to see that the equivalence is induced by direct sum of the maps $ \Lcont(i_{p,\sharp}) $, it suffices to show the assignment $ i \mapsto \ilowersharp $ respects composition: this follows from that its right adjoint $\iupperstar$, which as a restriction of functors, respects composition.
\end{proof}

\begin{corollary}\label{cor:splitting_for_exodromic_stratified_topoi}
	Let $ (\Xcal,P) $ be an exodromic stratified \topos and $ \Ecal $ be a dualizable \category.
	If $ P $ is finite or $ L $ is fintary, then there is a natural equivalence
	\begin{equation*}
		\Lcont(\ConsP(\Xcal;\Ecal)) \equivalent \Directsum_{p \in P} \Lcont(\Fun(\Piinfty(\Xcal_p),\Ecal)) \period
	\end{equation*}
	Here, $ \Piinfty(\Xcal_p) $ denotes the shape of the \topos $ \Xcal_p $.
\end{corollary}

\begin{proof}
	Since $ (\Xcal,P) $ is exodromic and $ \Ecal $ is dualizable, \cite[\href{https://arxiv.org/pdf/2401.12825\#equation.4.1.15}{Corollary 4.1.15}]{arXiv:2401.12825} shows that we have an exodromy equivalence with $ \Ecal $-coefficients
	\begin{equation*}
		\ConsP(\Xcal;\Ecal) \equivalent \Fun(\Piinfty(\Xcal,P),\Ecal) \period
	\end{equation*}
	Moreover, \cite[\href{https://arxiv.org/pdf/2401.12825\#equation.3.1.17}{Corollary 3.1.17}]{arXiv:2401.12825} shows that for each $ p \in P $, we have a natural equivalence 
	\begin{equation*}
		\Piinfty(\Xcal,P)_p \equivalent \Piinfty(\Xcal_p) \period
	\end{equation*}
	Hence the claim is immediate from \Cref{cor:splitting_for_functors} applied $ \Ccal = \Piinfty(\Xcal,P) $.
\end{proof}

We conclude this subsection with some applications to $ \Aup $-theory and topological Hochschild homology.

\begin{notation}
	Let $ X $ be an anima.
	We write $ \Lcal X = \Map(\Sup^1,X) $ for the anima of free loops on $ X $.
	Recall that the \textit{$ \Aup $-theory} of $ X $ is the \Ktheory spectrum
	\begin{equation*}
		\Aup(X) \colonequals \Kup(\Fun(X,\Sp)^{\upomega}) \equivalent \Kcont(\Fun(X,\Sp)) \in \Sp \period
	\end{equation*}
\end{notation}

\begin{corollary}\label{cor:K-theory_and_THH}
	Let $ (\Xcal,P) $ be an exodromic stratified \topos.
	Then there are natural equivalences 
	\begin{equation*}
		\Kcont(\ConsP(\Xcal;\Sp))\equivalent \Directsum_{p \in P} \Aup(\Piinfty(\Xcal_p)) 
		\quad \text{and} \quad
		\THH^{\cont}(\ConsP(\Xcal;\Sp))\equivalent \Directsum_{p \in P}\Sigma_+^\infty \Lcal \Piinfty(\Xcal_p) \period
	\end{equation*}
\end{corollary}

\begin{proof}
	The first equivalence is immediate from \Cref{cor:splitting_for_exodromic_stratified_topoi} and the fact that \Ktheory is finitary.
	For the second equivalence, note that since $ \THH $ is finitary, by \Cref{cor:splitting_for_exodromic_stratified_topoi}, it suffices to compute $\THH(\Fun(\Piinfty(\Xcal_p);\Sp))$ for each $ p\in P $. 
	Now we recall the fundamental calculation (see, for example, \cite[\href{https://publications.mfo.de/bitstream/handle/mfo/3637/OWR_2018_15.pdf?sequence=1&isAllowed=y\#page=64}{Corollary 5}]{MR3941522:Land}) that for an anima $ X $, we have a natural equivalence of spectra
	\begin{equation*}
		\THH^{\cont}(\Fun(X,\Sp)) \equivalent \Sigma_+^\infty \Lcal X \period \qedhere
	\end{equation*}
\end{proof}

\begin{nul}
	It may be surprising that $\THH^{\cont} $ of the \category of constructible sheaves only sees free loops traveling inside each stratum, but not free loops traveling through different strata. 
\end{nul}

\begin{remark}[(the lattice conjecture)] \label{rem:deducing_lattice_conjecture_from_strata}
	There are more complicated localizing invariants whose values on \categories of constructible sheaves are interesting.
	For example, periodic cyclic homology $\mathrm{HP}$ and topological \Ktheory $\Kup^{\mathrm{top}}$. 
	Let us briefly mention that Blanc's \textit{lattice conjecture} \cite[\href{https://arxiv.org/pdf/1211.7360\#prop.1.7}{Conjecture 1.7}]{MR3477639} is formulated with these invariants. 
	More precisely, for \smash{$\Ccal\in\Catperf$} which is $\CC$-linear, it asks if the Chern character
	\begin{equation*}
		\Kup^{\mathrm{top}}(\Ccal) \tensor \CC \longrightarrow \mathrm{HP}(\Ccal)
	\end{equation*}
	is an equivalence.
	Given that the conjecture has been proven for $\Ccal$ the \category of (compact objects in) $\Mod_\CC$-valued local systems on a certain class of topological spaces \cite[\href{https://arxiv.org/pdf/2102.01566\#prop.6.16}{Proposition 6.16}]{arxiv:2102.01566}, our result shows that the same is true for the \category of $\Mod_\CC$-valued constructible sheaves on certain stratified topological spaces.
	More precisely, for an exodromic stratified topological space $(X,P)$ with finite $P$, such that each stratum $X_p$ falls into the above class (whose \category of local systems $\LC(X_p;\Mod_\CC)^\upomega$ satisfies the lattice conjecture), the \category of constructible sheaves $\ConsP(X;\Mod_\CC)^\upomega$ also satisfies the lattice conjecture. To show this, given our results above, it suffices to note that both sides decompose into direct sums, and the Chern character map induces equivalences between corresponding summands by assumption.	Compare \cite[\href{https://arxiv.org/pdf/2102.01566\#prop.6.18}{Remark 6.18}]{arxiv:2102.01566}. 
\end{remark}


\subsection{Small \texorpdfstring{$\infty$}{∞}-categories}\label{subsec:small_categories}

We now explain splitting results for small \categories of constructible sheaves `with compact stalks'.
Since a general \topos need not have enough points, it is better to formulate a definition using locally constant sheaves that agrees with the stalk-wise definition for exodromic \topoi.

\begin{definition}
	Given \atopos $ \Xcal $ and a presentable \category $ \Ecal $, write $ \Gammaupperstar_{\Xcal} \colon \Ecal \to \Sh(\Xcal;\Ecal) $ for the constant sheaf functor, i.e., the left adjoint to global sections.
	Let $ \Ecal_0 \subset \Ecal $ be a full subcategory. 
	We write
	\begin{equation*}
		\LC(\Xcal;\Ecal_0) \subset \Sh(\Xcal;\Ecal)
	\end{equation*}
	for the full subcategory spanned by those objects $ F $ such that there exists an effective epimorphism $ \coprod_{i \in I} U_i \surjection 1_{\Xcal} $ in $ \Xcal $ and for each $ i \in I $ there exists an object $ e_i \in \Ecal_{0} $ and an equivalence
	\begin{equation*}	
		\jupperstar_{i}(F) \equivalent \Gammaupperstar_{\Xcal_{/U_i}}(e_i)
	\end{equation*}
	in $ \Sh(\Xcal_{/U_i};\Ecal) $.
	Here, \smash{$ \jupperstar_{i} \colon \Sh(\Xcal;\Ecal) \to \Sh(\Xcal_{/U_i};\Ecal) $} denotes the pullback functor induced by \smash{$ U_i \cross (-) \colon \Xcal \to \Xcal_{/U_i} $}.
\end{definition}

\begin{definition}
	Let $ (\Xcal,P) $ be a stratified \topos, let $ \Ecal $ be a presentable \category, and let $ \Ecal_0 \subset \Ecal $ be a full subcategory.
	We write 
	\begin{equation*}
		\ConsP(\Xcal;\Ecal_0) \subset \Sh(\Xcal;\Ecal)
	\end{equation*}
	for the full subcategory spanned by those objects $ F $ such that for each $ p \in P $, the restriction $ \iupperstar_p(F) $ is the in full subcategory $ \LC(\Xcal_p;\Ecal_0) \subset \Sh(\Xcal_p;\Ecal) $.
\end{definition}

\begin{nul}
	If $ (\Xcal,P) $ is an exodromic stratified \topos and $ \Ecal $ is a compactly generated presentable \category with compact objects \smash{$ \Ecal^\upomega \subset \Ecal $}, then the exodromy equivalence 
	\begin{align*}
		\ConsP(\Xcal;\Ecal) &\equivalent \Fun(\Piinfty(\Xcal,P),\Ecal) \\ 
	\intertext{restricts to an equivalence} 
		\ConsP(\Xcal;\Ecal^\upomega) &\equivalent \Fun(\Piinfty(\Xcal,P),\Ecal^\upomega) \period
	\end{align*}
	In the setting of hypersheaves on an exodromic stratified space $ (X,P) $ with locally weakly contractible strata, the full subcategory $ \ConsP(\Xcal;\Ecal^\upomega) \subset \ConsP(\Xcal;\Ecal) $ coincides with the full subcategory spanned by those objects whose stalks are compact objects of $ \Ecal $.
	Compare to \Cite[\href{https://arxiv.org/pdf/2401.12825\#equation.5.4.7}{Observation 5.4.7}]{arXiv:2401.12825}. 
\end{nul}

\begin{notation}
	In the rest of this subsection, we fix a compactly generated dualizable \category $ \Ecal $, and consider the subcategory of compact objects \smash{$ \Ecal^\upomega \in \Catperf $}. 
\end{notation}

\begin{lemma}\label{lem:splitting_for_small_categories}
	Let $ s \colon \Ccal \to P $ be a functor from a small \category to a poset. 
	Let $ Z \subset  P $ be a closed poset and $ U \subset P $ be the open complement.
	\begin{enumerate}
		\item The functors from \Cref{ex:recollement_for_functors_out_of_a_layered_category} restrict to a (one-sided) split Verdier sequence in $\Catperf$
		\begin{equation*}
			\begin{tikzcd}[sep=6em]
				\Fun(\Ccal_Z,\Ecal^\upomega) \arrow[r, "\ilowerstar"', shift right=2ex, hooked] & \Fun(\Ccal,\Ecal^\upomega) \arrow[l, "\iupperstar"'] \arrow[r, "\jupperstar"'] & \Fun(\Ccal_U,\Ecal^\upomega) \period \arrow[l, shift right=2ex, hooked', "\jlowersharp"']
			\end{tikzcd}
		\end{equation*}

		\item The maps $ L(\iupperstar) $ and $ L(\jupperstar) $ induce an equivalence
		\begin{equation*}
			L(\Fun(\Ccal,\Ecal^\upomega)) \equivalence L(\Fun(\Ccal_Z,\Ecal^\upomega))\oplus L(\Fun(\Ccal_U,\Ecal^\upomega)) \period
		\end{equation*}

		\item If the poset $ P $ is finite, the maps $L(\iupperstar_p)$ induce a natural equivalence
		\begin{equation*}
			L(\Fun(\Ccal,\Ecal^\upomega)) \equivalence \Directsum_{p \in P} L(\Fun(\Ccal_p,\Ecal^\upomega)) \period 
		\end{equation*}
	\end{enumerate}
\end{lemma}

\begin{proof}
	For the first item, note that $ \Fun(\Ccal,\Ecal^\upomega) $ is a full subcategory of $ \Fun(\Ccal,\Ecal) $ which is characterized by a pointwise condition. 
	In particular, the restriction functors $\iupperstar$ and $\jupperstar$ preserve these full subcategories. 
	Moreover, by the formulas in items \eqref{lem:existence_of_recollements_for_functors_out_of_a_layered_category.2} and \eqref{lem:existence_of_recollements_for_functors_out_of_a_layered_category.3} of \Cref{lem:existence_of_recollements_for_functors_out_of_a_layered_category}, the pushforward functors $ \ilowerstar $ and $ \jlowersharp $ also preserve these full subcategories. 
	It follows that we have the desired adjunctions.
	Because $ \jupperstar $ has a fully faithful left adjoint $ \jlowersharp $, and $i_*$ is precisely the inclusion of the kernel of $\jupperstar$, this is a split Verdier sequence.

	The second item follows immediately from the fact that we have a split Verdier sequence. 
	Finally, the third item is proven by induction on the Krull dimension of $ P $ exactly as in the proof of \cref{cor:splitting_for_functors}; we omit the details.
\end{proof}

\begin{corollary} \label{cor:splitting_constructible_small_category_finite_poset}
	Let $(\Xcal,P)$ be an exodromic stratified \topos where $ P $ is a finite poset. 
	Then the maps $L(\iupperstar_p)$ induce a natural equivalence
	\begin{equation*}
		\ConsP(\Xcal,\Ecal^\upomega) \equivalent \bigoplus_{ p \in P } L(\Fun(\Piinfty(\Xcal_p),\Ecal^\upomega)) \period
	\end{equation*}
\end{corollary}

\begin{proof}
	Immediate from \cref{lem:splitting_for_small_categories}.
\end{proof}

\begin{remark}
	We discuss here the subtle differences in the \categories appearing in this subsection and the previous one. 
	Fix a stratified topological space $(X,P)$ which is hypercomplete as well as a compactly generated stable \category $\Ecal$.
	\begin{enumerate}		
		\item In general, the \category $ \ConsP(X;\Ecal^\upomega) $ is not (on the nose) the same as $\ConsP(X;\Ecal)^\upomega$, though they are both full subcategories of $\ConsP(X;\Ecal)$: the latter is defined to be the subcategory of compact objects in the presentable \category $\ConsP(X;\Ecal)$. 
		Hence we don't know if it is possible to deduce the statements for small \categories from the statements for large \categories or vice versa. 
		It would be very interesting to formulate finiteness conditions (as those in \cite{arXiv:2412.04745}) on the exit-path \category to ensure that these two \categories agree.
		
		\item The following example is taken from \cite[\href{https://arxiv.org/pdf/1604.00114\#thm.3.19}{Example 3.19}]{arXiv:1604.00114}. 
		In this case there is a natural inclusion from one \category to the other, but it is not an equivalence. 
		Consider the topological circle $ \Sup^1$ equipped with the trivial stratification (so constructible sheaves are locally constant shaves), and let the coefficient \category $ \Ecal $ be the derived \category $ \Mod_\CC $ of $\CC$-vector spaces. 
		Now both \categories can be considered as subcategories of the large \category of modules over the group ring $ \CC[\Omega S^1] = \CC[\ZZ] $.
		\begin{enumerate}
			\item[\textbullet] The category $\LC(\Sup^1;\Mod_\CC^\upomega)$ is the subcategory of objects whose underlying complexes of $\CC$-vector spaces are perfect.

			\item[\textbullet] The \category $\LC(\Sup^1;\Mod_\CC)^\upomega$ is the subcategory of perfect complexes of $\CC[\ZZ]$-modules.
		\end{enumerate}
 		Note that in this case, there is an inclusion $\LC(\Sup^1;\Mod_\CC^\upomega)\subset\LC(\Sup^1;\Mod_\CC)^\upomega$. 
 		However, the inclusion map induces the zero map on the Grothendieck group $\Kup_0$.

 		\item When the poset $ P $ is infinite, we don't have a general statement about localizing invariants of $\ConsP(X;\Ecal^\upomega)$, or equivalently $\Fun(\Piinfty(X,P),\Ecal^\upomega)$. 
 		One can write this \category as an cofiltered limit in $\Catperf$:
 		\begin{equation*}
 			\Fun(\Piinfty(X,P),\Ecal^\upomega)\equivalent{\lim_{Q \in \Subfin(P)^{\op}} \Fun(\Piinfty(X,P)_Q,\Ecal^\upomega)}\period
 		\end{equation*}	
 		But we don't know if we can commute localizing invariants past such cofiltered limits.
 		In general, this is a delicate problem; let us mention that Efmiov \cite{arXiv:2502.04123} has studied this in the setting of dualizable \categories. 

 		In the other direction, if we assume the poset $ P $ can be filtered by a sequence of closed (or open) subsets, then we can `renormalize' the \category of constructible sheaves and describe its localizing invariants, as follows.
 		(This is a common practice in applications, see for example \cite[\href{https://arxiv.org/pdf/math/0202150\#subsection.2.2}{\S2.2}]{MR2142332}.)
 	\end{enumerate}
\end{remark}

\begin{notation}\label{ntn:renormalized_Cons}
	Let $(\Xcal,P)$ be an exodromic stratified \topos. 
	Let $ \Zcal $ be a filtered family of \textit{finite} closed subposets of $ P $ such that $ \Union_{Z \in \Zcal} Z = P $.
	We define the subcategory of constructible sheaves valued in $\Ecal^\upomega$ which are $ * $-extended from some $ \Xcal_{Z} $ as
	\begin{equation*}
		\ConsP^{\Zren}(\Xcal;\Ecal^\upomega) \colonequals \colim_{Z \in \Zcal} \Cons_{Z}(\Xcal_{Z};\Ecal^\upomega) \period
	\end{equation*}
	The colimit is formed along the $ * $-pushforward functors.
\end{notation}

\begin{lemma}
	In the setting of \Cref{ntn:renormalized_Cons}, if $ L $ is a finitary localizing invariant, then
	\begin{equation*}
		L(\ConsP^{\Zren}(\Xcal;\Ecal^\upomega)) \equivalent \bigoplus_{ p \in P } L(\Fun(\Piinfty(\Xcal_p),\Ecal^\upomega)) \period
	\end{equation*}
\end{lemma}

\begin{proof}
	Since $L$ is finitary, the statement reduces to the case of finite poset (\Cref{cor:splitting_constructible_small_category_finite_poset}).
\end{proof}


\DeclareFieldFormat{labelnumberwidth}{#1}
\printbibliography[keyword=alph, heading=references]
\DeclareFieldFormat{labelnumberwidth}{{#1\adddot\midsentence}}
\printbibliography[heading=none, notkeyword=alph]

\end{document}